\documentclass[reqno, 12pt]{amsart}
\pdfoutput=1
\makeatletter
\let\origsection=\section \def\section{\@ifstar{\origsection*}{\mysection}} 
\def\mysection{\@startsection{section}{1}\z@{.7\linespacing\@plus\linespacing}{.5\linespacing}{\normalfont\scshape\centering\S}}
\makeatother        
\usepackage{amsmath,amssymb,amsthm}    
\usepackage{mathabx}\usepackage{wasysym}
\usepackage{mathrsfs}
\usepackage{bbm, accents} 

\usepackage{xcolor}  	
\usepackage[backref]{hyperref}
\hypersetup{
	colorlinks,
    linkcolor={red!60!black},
    citecolor={green!60!black},
    urlcolor={blue!60!black},
}

\usepackage{tikz}
\usetikzlibrary{calc,decorations.pathmorphing}

\usepackage{bookmark}

\usepackage[abbrev,msc-links,backrefs]{amsrefs} 
\usepackage{doi}

\renewcommand{\PrintDOI}[1]{\doi{#1}}

\usepackage[T1]{fontenc}
\usepackage{lmodern}

\usepackage[english]{babel}
\usepackage{microtype}

\usepackage{graphicx}              \usepackage{caption}
\usepackage{subcaption}
                          
\def\gs{\geqslant}
\def\ls{\leqslant}
\def\EE#1{E(#1)}
\def\VV#1{V(#1)}

\let\sm=\setminus

\newtheorem{fact}{Fact}\newtheorem{thm}[fact]{Theorem}
\newtheorem{cor}[fact]{Corollary}

\usepackage{geometry}
\numberwithin{equation}{section}

\begin{document}
\title[Decomposing graphs into forests]
{Nash-Williams' theorem on decomposing graphs into forests}
\author{Christian Reiher}
\address{Fachbereich Mathematik, Universit\"at Hamburg, Bundesstra\ss e 55, 20144 Hamburg, Germany}
\email{Christian.Reiher@uni-hamburg.de}
\author{Lisa Sauermann}
\address{Department of Mathematics, Stanford University, 450 Serra Mall, 
Building 380, Stanford CA 94305, USA}
\email{lsauerma@stanford.edu}         
\begin{abstract}
We give a simple graph-theoretic proof of a classical result due to {\sc C. St. J. A. Nash-Williams} on covering graphs by forests. Moreover we derive a slight generalisation of this statement where some edges are preassigned to distinct forests. 
\end{abstract}

\keywords{Nash-William's theorem, forest covering, arboricity}

\maketitle

\section{Introduction}\label{sec:intro}
All graphs considered in this note are finite. The sets of vertices and edges of a graph 
$G$ are denoted by $V(G)$ and $E(G)$, respectively. 
The {\it restriction} of a graph $G=(V, E)$ to a subset $X$ of $V$, i.e., the graph 
on $X$ whose edges are precisely those members of $E$ both of whose ends belong to $X$, 
is indicated by $G|X$. When $G$ is clear from the context, we write $e(X)$ for the number 
of edges of that graph. For any integer $r\ge 0$ we set $[r]=\{1, 2, \ldots, r\}$.

The following result is due to {\sc C. St. J. A. Nash--Williams} (see \cite{NW2}, and the 
related articles \cites{NW1, Tut} as well as \cite{China} for another simple proof).

\begin{thm} \label{thm:1} If $G=(V, E)$ is a graph, and $r\ge 0$ is an integer such that for all 
nonempty subsets $X$ of $V$ one has $e(X)\le r(|X|-1)$, then there exists a 
partition $E=E_1\cup E_2\cup\ldots\cup E_r$ such that $(V, E_i)$ is a forest for 
$i\in[r]$.
\end{thm}

It is plain that the sufficient condition for such a partition to exist given here is 
also necessary. Also, the cases $r=0$ and $r=1$ of this statement are immediate, and the 
case $r=2$ was recently posed at the All-Russian Mathematical Olympiad, \cite{Russen}. 
That case  can be dealt with by some peculiar tricks not discussed here and that do not 
straightforwardly generalize to $r>2$, but which nevertheless motivated us to reprove 
the general case independently. In fact, it turned out that our arguments for the case 
$r=2$ yielded slightly more, namely that for any two distinct edges of $G$ there 
exists such a partition in which one of the edges belongs to $E_1$ while the other one 
belongs to $E_2$. Hence one might guess:

\begin{cor}\label{cor:2}
	Given a graph $G=(V, E)$, an integer $r\gs 0$ such that for all nonempty 
	subsets $X$ of $V$ one has $e(X)\ls r(|X|-1)$, and moreover a sequence 
	$e_1, e_2, \ldots, e_{r}$ of  distinct edges of $G$, there exists a partition 
	$E=E_1\cup E_2\cup\ldots\cup E_r$ such that $e_i\in E_i$ for $i\in[r]$ and $(V, E_i)$ 
	is a forest for all $i\in[r]$.
\end{cor}

As we shall see in Section~\ref{sec:3}, this can in fact be derived from Theorem~\ref{thm:1}.

All statements and arguments contained in this article are valid irrespective of whether 
multiple edges are allowed to occur in our graphs or not.

\section{Proving Theorem~\ref{thm:1}} \label{sec:2}

In this section we give a simple proof of Theorem~\ref{thm:1} that is, to the best of our 
knowledge, new. For this purpose we need some preparation. Let $T$ be a 
forest, $k\geq 2$ an integer, and $T_1, T_2, \ldots ,T_k$ 
mutually vertex disjoint, connected subgraphs of $T$. 
Obviously $T_1, T_2, \ldots T_k$ are trees and each of them 
is contained in exactly one component of $T$. We call $T_i$ 
{\it isolated} in $T$ if there is no $T_j$ with $j\neq i$, that is 
contained in the same component of $T$ as $T_i$. Furthermore we call 
$T_i$ {\it peculiar} in $T$ if there is an edge 
$e_i\in E(T)\sm\bigl(E(T_1)\cup \dots \cup E(T_k)\bigr)$ incident with a 
vertex of $T_i$, such that $T_i$ is isolated in $T-e_i$ 
(the reader may notice that one gets an equivalent notion without demanding $e_i$ to be 
incident with a vertex of $T_i$).

\begin{fact}\label{fact:3}
At least two of the subgraphs $T_1$, $T_2$, \dots ,$T_k$ 
are isolated or peculiar in $T$.
\end{fact}

\begin{proof} 
Otherwise take a counterexample where $T$ has as few vertices as possible. 
If no component of $T$ contains two or more of the $T_i$, 
then each them is isolated and we are done. In the remaining cases $T$ 
is a tree. Consider any leaf $x$ of $T$. As we cannot produce a 
smaller counterexample by deleting $x$ thus, there has to be some $T_i$ consisting 
solely of~$x$, and the edge of $T$ incident with $x$ witnesses that this 
$T_i$ is peculiar. Because of $k\geq 2$ the tree $T$ has at least 
two vertices and, consequently, at least two leaves. Applying the foregoing argument to 
any two leaves of $T$ we see that at least two of the trees 
$T_1$, $T_2$, \dots , and $T_k$ are peculiar. 
\end{proof}

\begin{proof}[Proof of Theorem~\ref{thm:1}.]
Arguing indirectly we choose a graph $G=(V,E)$ and an integer $r\gs 0$ 
contradicting Theorem~\ref{thm:1} with $\vert E\vert$ minimal. Then $\vert E\vert\neq 0$. 
Let $e=ab$ be an arbitrary edge of $G$.

Because of the choice of $G$, there is at least one partition 
$E\sm\{e\}=E_1'\cup E_2'\cup \dots \cup E_r'$ such that $(V, E_i')$ is a forest for 
$i=1,2,\dots ,r$. For each of these partitions we consider the component $(C, E_C)$ 
of $(V, E_1')$ containing the vertex $a$. 
From now on let 
\[
	E\sm\lbrace e\rbrace=E_1'\cup E_2'\cup \dots \cup E_r'
\]
be one of these 
partitions with $\vert C\vert$ minimum. Let $\overline{C}=\EE{G\vert C}$.

If $b\not\in C$, the partition $E=(E_1'\cup \lbrace e\rbrace)\cup E_2'\cup \dots \cup E_r'$ 
would satisfy all conditions of Theorem~\ref{thm:1}, consequently $b\in C$ and $e\in \overline{C}$. 
Therefore
\[
	\vert E_1'\cap \overline{C}\vert+\dots+\vert E_r'\cap \overline{C}\vert
	<\vert \overline{C}\vert=e(C)\ls r(\vert C\vert-1)\,.
\]
Thus, there is an $i\in [r]$ with $\vert E_i'\cap \overline{C}\vert<\vert C\vert-1$. 
This implies, that $(C, E_i'\cap \overline{C})$ is not connected. 
Because of the definition of $(C, E_C)$ we have $i\neq 1$, so we can w.l.o.g. assume~$i=2$.

Let $D_1$, \dots, $D_k$ be the connected components of 
$(C, E_2'\cap \overline{C})$, where obviously $k\geq 2$. Thus, 
$D_1$, \dots, $D_k$ are mutually vertex disjoint, connected subgraphs 
of the forest $(V, E_2')$. We define isolation  and peculiarity in $(V, E_2')$ 
as applying to these subgraphs. By Fact~\ref{fact:3} at least two of the subgraphs 
$D_1$, \dots, $D_k$ are isolated or peculiar in $(V, E_2')$, and at 
most one of them contains $a$. Let w.l.o.g. $D_1$ be isolated or peculiar 
in $(V, E_2')$ and $a\not\in \VV{D_1}$.

If $D_1$ is peculiar in $(V, E_2')$, there is an edge
\[
	e_1\in E_2'\sm\bigl(\EE{D_1}\cup \dots \cup \EE{D_k}\bigr)
\]
incident with a vertex $v_1$ of $D_1$, such that $D_1$ is 
isolated in $(V, E_2'\sm\lbrace e_1\rbrace)$. Notice that
\[
	\EE{D_1}\cup \dots \cup \EE{D_k}=\overline{C}\cap E_2'
\]
yields $e_1\not\in \overline{C}$, wherefore $e_1$ connects $v_1$ and a vertex not in $C$.

If $D_1$ is isolated in $(V, E_2')$, let $v_1$ be an arbitrary vertex of $D_1$.

We consider the uniquely determined path from $a$ to $v_1$ in the tree $(C, E_C)$. 
Since $a\not\in \VV{D_1}$ and $v_1\in \VV{D_1}$, this path 
contains an edge $e_d$ connecting a vertex of $D_1$ with a vertex of some 
$D_i$ with $i\neq 1$.

First, we consider the case when $D_1$ is isolated in $(V, E_2')$. 
Then the graph $(V, E_2'\cup \lbrace e_d\rbrace)$ is a forest, because $e_d$ 
connects different components of $(V, E_2')$. Obviously, the graph 
$(V, E_1'\sm \lbrace e_d\rbrace)$ is also a forest and its component including $a$ 
is a subgraph of $(C, E_C)$. But this subgraph does not contain $v_1$ and has 
therefore a number of vertices smaller than $\vert C\vert$. This contradicts
\[
	E\sm\lbrace e\rbrace=(E_1'\sm \lbrace e_d\rbrace)\cup (E_2'
	\cup \lbrace e_d\rbrace)\cup \dots \cup E_r'
\]
being one of the partitions considered at the beginning.

For the second case let $D_1$ now be peculiar in $(V, E_2')$. Then $D_1$ is isolated in 
$(V, E_2'\sm \lbrace e_1\rbrace)$ and the graph 
$\bigl(V, (E_2'\sm\lbrace e_1\rbrace)\cup \lbrace e_d\rbrace\bigr)$ is therefore a forest. 
On the other hand the graph $(V, E_1'\sm\lbrace e_d\rbrace)$ is also a forest and a subgraph 
of the forest $(V, E_1')$. Because $e_d$ belongs to the component $(C, E_C)$ of $(V, E_1')$, 
the graph $(V, E_1'\sm\lbrace e_d\rbrace)$ has two components being subgraphs of $(C, E_C)$, 
one of which contains $a$ and the other one $v_1$. The edge $e_1$ connects~$v_1$ and a vertex 
not in $C$, consequently the graph 
$\bigl(V, (E_1'\sm\lbrace e_d\rbrace)\cup \lbrace e_1\rbrace\bigr)$ is 
a forest and its component including $a$ is equal to the component of 
$(V, E_1'-\lbrace e_d\rbrace)$ including $a$. 
Thus, its number of vertices is smaller than $\vert C\vert$. 
This contradicts 
\[
	E\sm\lbrace e\rbrace=\bigl((E_1'\sm \lbrace e_d\rbrace)\cup \lbrace e_1\rbrace\bigr)
	\cup \bigl((E_2'\sm\lbrace e_1\rbrace)\cup \lbrace e_d\rbrace\bigr)
	\cup \dots \cup E_r'
\]
being one of the partitions considered at the beginning.

\begin{figure}[ht]
\centering
\begin{tikzpicture}[scale=1.2]
	\coordinate (x) at (0.1,1);
	\coordinate (y) at (0, 5);
	
	\coordinate (r1) at (2, 5.5);
	\coordinate (r2) at (2, 4.5);
	\coordinate (r3) at (5, 5.5);
	\coordinate (r4) at (5.8,4.5);
	\coordinate (r5) at (1,1.4);
	\coordinate (a)  at (1,0.6);
    \coordinate (r6) at (2.2,1.5);
    \coordinate (r7) at (2.1,0.8);	
    \coordinate (r8) at (4,0.6);	
    \coordinate (r9) at (5.6,0.9);	
    \coordinate (v1) at (5.2,1.5);		
		
	\draw[line width=1pt] (y) -- (r2);
	\draw[line width=1pt] (r1) -- (r3);
	\draw[line width=1pt] (r2) -- (r3);
	\draw[line width=1pt] (r4) -- (r2);
	\draw[line width=1pt] (x) -- (r5) -- (a) -- (r7) -- (r6);
	\draw[line width=1pt] (r7) -- (r8) -- (r9);
	\draw[line width=1pt] (r8) -- (v1);
	
	\draw[green, line width=1pt] (y) -- (r1) -- (r2);
	\draw[green, line width=1pt] (r1) -- (r4) -- (r3);
	\draw[green, line width=1pt] (r9) -- (v1) -- (r6) -- (r8);
	\draw[green, line width=1pt] (a) -- (x) -- (y);
	\draw[green, line width=1pt] (r5) -- (r6);
	\draw[green, line width=1pt] (v1) -- (r4);
	\draw[green, line width=1pt] (x) -- (r7);
	
	\foreach \i in {x, y, r1, r2, r3, r4, r5, a, r6, r7, r8, r9, v1}
			\fill  (\i) circle (2pt);
		
	\node at (5.9,3) {$e_1$};
	\node at (3.2, 0.3) {$e_d$};
	\node at (1,0.2) {$a$};
	\node at (5.6, 1.6) {$v_1$};
	\node at (-1.3, 3) {$C$};		
\end{tikzpicture}
\caption{$D_1$ is peculiar}
\label{fig:1}
\end{figure}
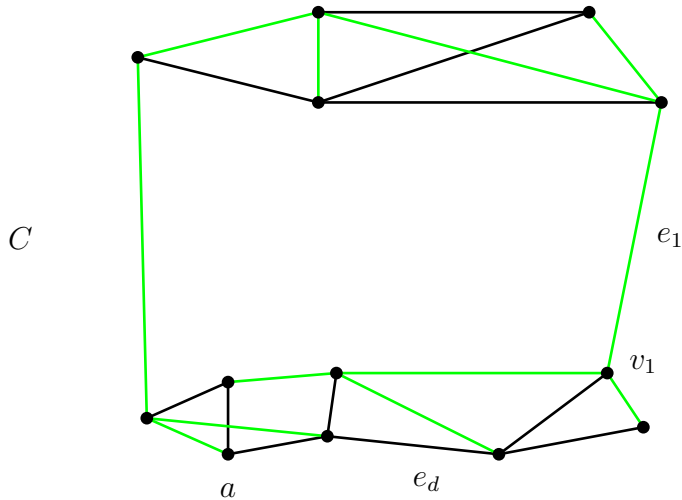

Since we have obtained a contradiction in each of the two cases, our assumption must 
have been wrong and Theorem~\ref{thm:1} is true.
\end{proof}

\section{Deducing Corollary~\ref{cor:2}} \label{sec:3}
 
The strengthening given by Corollary~\ref{cor:2} will now be deduced from 
Theorem~\ref{thm:1} by means of a short argument.

\begin{proof}[Proof of Corollary~\ref{cor:2}] Let $G=(V,E)$, $r\ge 0$, and $e_1, \dots ,e_r$ 
be as in Corollary~\ref{cor:2}. We call an integer $0\ls r'\ls r$ {\it restrained} 
if there is a partition $E=E_1\cup E_2\cup \dots \cup E_r$ such that 
$e_i\in E_i$ for $i\in [r']$ and $(V, E_i)$ is acyclic for all $i\in [r]$. 
By Theorem~\ref{thm:1} the integer $0$ is restrained. 
Corollary~\ref{cor:2} is equivalent to $r$ being restrained. 
It is therefore sufficient to prove the following: if an integer $k$ with $0\ls k\ls r-1$ 
is restrained, then the integer $k+1$ is also restrained.

Let $E=E_1\cup E_2\cup \dots \cup E_r$ be a partition such that $e_i\in E_i$ for $i\in [k]$ 
and $(V, E_i)$ is acyclic for all $i\in [r]$. If $e_{k+1}\in E_{k+1}$, we are done. 
We can therefore assume $e_{k+1}\in E_\ell$ with $\ell\neq k+1$. Obviously we can assume that 
the two vertices of $e_{k+1}$ belong to the same component of $(V, E_{k+1})$.
        
The two vertices of $e_{k+1}$ belong to different components of the forest 
$(V, E_\ell\sm\lbrace e_{k+1}\rbrace)$. Therefore the uniquely determined path between 
these two vertices in the forest $(V, E_{k+1})$ contains vertices of different 
components of $(V, E_\ell\sm\lbrace e_{k+1}\rbrace)$. Hence, there is an edge $e\in E_{k+1}$ 
of this path connecting vertices of different components of $(V, E_\ell\sm\lbrace e_{k+1}\rbrace)$. 
Then the graph $\bigl(V, (E_\ell\sm\lbrace e_{k+1}\rbrace)\cup \lbrace e\rbrace\bigr)$ is a forest. 
On the other hand the graph 
$\bigl(V, (E_{k+1}\sm\lbrace e\rbrace)\cup \lbrace e_{k+1}\rbrace\bigr)$ 
is by the definition of $e$ also a forest.

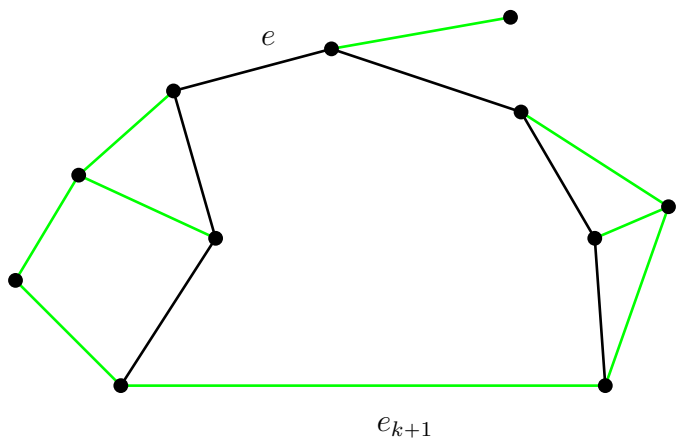
\begin{figure}[ht]
\centering
\begin{tikzpicture}[scale=1.4]
		
	\coordinate (r1) at (0, 1);
	\coordinate (r2) at (0.6, 2);
	\coordinate (r3) at (1, 0);
	\coordinate (r4) at (1.9,1.4);
	\coordinate (r5) at (1.5,2.8);
    \coordinate (r6) at (5.6,0);
    \coordinate (r7) at (5.5,1.4);	
    \coordinate (r8) at (6.2,1.7);	
    \coordinate (r9) at (4.8,2.6);	
    \coordinate (r10) at (3,3.2);
    \coordinate (r11) at (4.7,3.5);		
		
	\draw[line width=1pt] (r3) -- (r4) -- (r5) -- (r10) -- (r9) -- (r7) -- (r6);	
	\draw[green, line width=1pt] (r9) -- (r8) -- (r6) -- (r3) -- (r1) -- (r2) -- (r4);
	\draw[green, line width=1pt] (r2) -- (r5);
	\draw[green, line width=1pt] (r10) -- (r11);
	\draw[green, line width=1pt] (r7) -- (r8);

	\foreach \i in {r1, r2, r3, r4, r5, r6, r7, r8, r9, r10, r11}
			\fill  (\i) circle (2pt);
		
	\node at (2.4,3.3) {$e$};
	\node at (3.7, -0.4) {$e_{k+1}$};
\end{tikzpicture}
\caption{Switching $e$ and $e_{k+1}$}
\label{fig:2}
\end{figure}

Thus, the partition gained from $E=E_1\cup E_2\cup \dots \cup E_r$ by substituting 
$E_\ell$ by ${(E_\ell\sm\lbrace e_{k+1}\rbrace )\cup \lbrace e\rbrace}$ and $E_{k+1}$ by 
$(E_{k+1}\sm\lbrace e\rbrace)\cup \lbrace e_{k+1}\rbrace$ fulfills all conditions for 
$k+1$ being restrained.
\end{proof}

\end{document}